\documentclass[9pt,oneside,english]{amsart}
\usepackage[T1]{fontenc}
\usepackage[latin9]{inputenc}
\usepackage[a4paper]{geometry}
\usepackage{textcomp}
\usepackage{amsthm}
\usepackage{amstext}
\usepackage{amssymb}

\makeatletter
\numberwithin{equation}{section}
\numberwithin{figure}{section}
\theoremstyle{plain}
\newtheorem{thm}{\protect\theoremname}[section]
  \theoremstyle{remark}
  \newtheorem{rem}[thm]{\protect\remarkname}
  \theoremstyle{plain}
  \newtheorem{prop}[thm]{\protect\propositionname}
  \theoremstyle{plain}
  \newtheorem{cor}[thm]{\protect\corollaryname}
  \theoremstyle{plain}
  \newtheorem{lem}[thm]{\protect\lemmaname}

\usepackage{xypic, pstricks}
\theoremstyle{definition}
\newtheorem{parn}{}[subsection]

\makeatother

\usepackage{babel}
  \providecommand{\corollaryname}{Corollary}
  \providecommand{\lemmaname}{Lemma}
  \providecommand{\propositionname}{Proposition}
  \providecommand{\remarkname}{Remark}
\providecommand{\theoremname}{Theorem}

\begin{document}
 \author{Adrien Dubouloz} 
\address{Adrien Dubouloz, Institut de Math\'ematiques de Bourgogne, Universit\'e de Bourgogne, 9 avenue Alain Savary - BP 47870, 21078 Dijon cedex, France}  \email{Adrien.Dubouloz@u-bourgogne.fr}
\author{David R. Finston}
\address{Department of Mathematical Sciences, New Mexico State University, Las Cruces, New Mexico 88003} 
\email{dfinston@nmsu.edu}  
\thanks{Research supported in part by NSF Grant OISE-0936691 and ANR Grant 08-JCJC-0130-01}

\psset{unit=0.8}
\newsavebox{\Xalone}
\savebox{\Xalone}{\pscurve(0,0)(1,0)(1.5,-0.2)(2,0)(2.5,0.2)(3,0)}

\newsavebox{\XaloneHole}
\savebox{\XaloneHole}
{\pscurve(0,0)(1,0)(1.2,-0.1)
\pscurve(1.7,-0.1)(2,0)(2.5,0.2)(3,0)
}

\newsavebox{\XtCzero}
\savebox{\XtCzero}{
\rput(0,0){\usebox{\Xalone}}
\rput(2,2){\usebox{\Xalone}}
\psline(3,0)(5,2)
\psline(0,0)(2,2)
\psline[linecolor=gray, linewidth=2px](1.5,-0.2)(3.5,1.8)
\psline[linecolor=gray, linewidth=0.5px](0.5,0)(2.5,2)
\psline[linecolor=gray, linewidth=0.5px](1,0)(3,2)
\psline[linecolor=gray, linewidth=0.5px](2,0)(4,2)
\psline[linecolor=gray, linewidth=0.5px](2.5,0.2)(4.5,2.2)
}

\newsavebox{\Sspace}
\savebox{\Sspace}{
\rput(0,0){\usebox{\XaloneHole}}
\rput(2,2){\usebox{\XaloneHole}}
\pscurve(1.2,-0.1)(1.25,-0.105)(1.3,-0.1)(1.7,0.1)(2.3,0.5)
(2.5,1.4)(3.4,1.8)(3.6,1.85)(3.65,1.87)(3.7,1.9)

\psline[linecolor=gray, linewidth=0.5px](0.5,0)(2.5,2)
\psline[linecolor=gray, linewidth=0.5px](1,0)(3,2)
\psline[linecolor=gray, linewidth=0.5px](2,0)(4,2)
\psline[linecolor=gray, linewidth=0.5px](2.5,0.2)(4.5,2.2)

}

\newsavebox{\Xpt}
\savebox{\Xpt}{
\psline(0,0)(3,0)
\rput(1.5,0){\textbullet}
\rput(1.5,-0.3){{\scriptsize $o$}}
\rput(3.4,0){{\scriptsize $X$}}
}

\newsavebox{\XCzero}
\savebox{\XCzero}{
\psline[linecolor=gray, linewidth=2px](1.5,0)(3.5,2)
\psline[linecolor=gray, linewidth=0.5px](0.5,0)(2.5,2)
\psline[linecolor=gray, linewidth=0.5px](1,0)(3,2)
\psline[linecolor=gray, linewidth=0.5px](2,0)(4,2)
\psline[linecolor=gray, linewidth=0.5px](2.5,0)(4.5,2)
\psline[linecolor=gray](3,0)(5,2)
\psline[linecolor=gray](0,0)(2,2)
\psline(0,0)(3,0)
\psline(2,2)(5,2)
\rput(2.8,0.9){{\scriptsize $C_0$}}

\rput(4,0.2){{\scriptsize $X\times C_0$}}
\psline{->}(1.5,-0.3) (1.5,-1)
\rput(2,-0.6){{\scriptsize $\mathrm{pr}_1$}}
\rput(0,-1.3){\usebox{\Xpt}}
}

\newsavebox{\Xspace}
\savebox{\Xspace}{
\pscurve(1.2,0)(2.3,0.5)(2.8,1.6)(3.7,2)
\psframe[fillstyle=solid, fillcolor=white, linecolor=white](1.48,0.05)(1.72,0.3)
\psframe[fillstyle=solid, fillcolor=white, linecolor=white](2.4,0.9)(2.7,1.15)
\psframe[fillstyle=solid, fillcolor=white, linecolor=white](3.3,1.8)(3.5,2)
\pscurve(1.7,0)(1.6,0.1) (3.4,1.9)(3.3,2)
\psline[linecolor=gray, linewidth=0.5px](0.5,0)(2.5,2)
\psline[linecolor=gray, linewidth=0.5px](1,0)(3,2)
\psline[linecolor=gray, linewidth=0.5px](2,0)(4,2)
\psline[linecolor=gray, linewidth=0.5px](2.5,0)(4.5,2)
\psline[linecolor=gray](3,0)(5,2)
\psline[linecolor=gray](0,0)(2,2)
\psline(0,0)(1.2,0)
\psline(1.7,0)(3,0)
\psline(5,2)(3.7,2)
\psline(3.3,2)(2,2)

\rput(3.5,0.2){{\scriptsize $\mathfrak{S}$}}
\rput(2.8,0.9){{\scriptsize $C_1$}}
\psline{->}(1.5,-0.3) (1.5,-1)
\rput(2.2,-0.6){{\scriptsize $\mathrm{pr}_1\circ \bar{\varphi}$}}
\rput(0,-1.3){\usebox{\Xpt}}

}

\newsavebox{\XC}
\savebox{\XC}{
\psline[linewidth=2px](1.5,0)(3.5,2)
\psline[linewidth=0.5px](0.5,0)(2.5,2)
\psline[linewidth=0.5px](1,0)(3,2)
\psline[linewidth=0.5px](2,0)(4,2)
\psline[linewidth=0.5px](2.5,0)(4.5,2)
\psline[linewidth=0.5px](3,0)(5,2)
\psline[linewidth=0.5px](0,0)(2,2)
\psline(0,0)(3,0)
\psline(2,2)(5,2)

\rput(2.65,0.9){{\scriptsize $C$}}
\rput(3.9,0.2){{\scriptsize $X\times C$}}
\psline{->}(1.5,-0.3) (1.5,-1)
\rput(2,-0.6){{\scriptsize $\mathrm{pr}_1$}}
\rput(0,-1.3){\usebox{\Xpt}}
}

\newsavebox{\XZ}
\savebox{\XZ}{
\psline(1.3,0)(1.6,0.2)(3.6,2.2)(3.3,2) 
\psline(1.7,0)(1.4,-0.2)(3.4,1.8)(3.7,2)
\psline[linewidth=2px](1.6,0.2)(3.6,2.2)
\psline[linewidth=2px](1.4,-0.2)(3.4,1.8)
\psline[linewidth=0.5px](0.5,0)(2.5,2)
\psline[linewidth=0.5px](1,0)(3,2)
\psline[linewidth=0.5px](2,0)(4,2)
\psline[linewidth=0.5px](2.5,0)(4.5,2)
\psline[linewidth=0.5px](3,0)(5,2)
\psline[linewidth=0.5px](0,0)(2,2)
\psline(0,0)(1.3,0)
\psline(1.7,0)(3,0)
\psline(2,2)(3.3,2)
\psline(3.7,2)(5,2)

\rput(3.5,0.2){{\scriptsize $Z$}}

}

\title{Proper twin-triangular $\mathbb{G}_{a}$-actions on $\mathbb{A}^{4}$
are translations }

\keywords{Additive group; twin-triangular derivation; geometric quotient; locally
trivial }

\subjclass[2000]{14R20, 14L30}
\begin{abstract}
An additive group action on an affine $3$-space over a complex Dedekind
domain $A$ is said to be twin-triangular if it is generated by a
locally nilpotent derivation of $A[y,z_{1},z_{2}]$ of the form $r\partial_{y}+p_{1}(y)\partial_{z_{1}}+p_{2}(y)\partial_{z_{2}}$,
where $r\in A$ and $p_{1},p_{2}\in A[y]$. We show that these actions
are translations if and only if they are proper. Our approach avoids
the computation of rings of invariants and focuses more on the nature
of geometric quotients for such actions. 
\end{abstract}

\dedicatory{Dedicated to Jim K. Deveney on the occasion of his retirement}

\maketitle

\section*{Introduction }

In 1968, Rentschler \cite{Rentschler1968} established in a pioneering
work that every algebraic action of the additive group $\mathbb{G}_{a}=\mathbb{G}_{a,\mathbb{C}}$
on the complex affine space $\mathbb{A}^{2}$ is triangular in a suitable
polynomial coordinate system. Consequently, every set-theoretically
free $\mathbb{G}_{a}$-action is a translation, in the sense that
$\mathbb{A}^{2}$ is equivariantly isomorphic to $\mathbb{A}^{1}\times\mathbb{A}^{1}$
where $\mathbb{G}_{a}$ acts by translations on the second factor.
An example due to Bass \cite{Bass1984} in 1984 shows that in higher
dimensions, $\mathbb{G}_{a}$-actions are no longer triangulable in
general, and Winkelmann \cite{Winkelmann1990} constructed in 1990
a set-theoretically free $\mathbb{G}_{a}$-action on $\mathbb{A}^{4}$
which is not a translation. The question about set-theoretically free
$\mathbb{G}_{a}$-actions on $\mathbb{A}^{3}$ was eventually settled
affirmatively first by Deveney and the second author \cite{Deveney1994a}
in 1994 under the additional assumption that the action is proper
and then in general by Kaliman \cite{Kaliman2004a} in 2004. 

For proper actions, the argument turns out to be much simpler than
the general one, the crucial fact being that combined with the flatness
of the algebraic quotient morphism $\pi:\mathbb{A}^{3}\rightarrow\mathbb{A}^{3}/\!/\mathbb{G}_{a}={\rm Spec}(\Gamma(\mathbb{A}^{3},\mathcal{O}_{\mathbb{A}^{3}})^{\mathbb{G}_{a}})$
which is obtained from dimension considerations, properness implies
that the action is locally trivial in the Zariski topology, i.e. that
$\mathbb{A}^{3}$ is covered by invariant Zariski affine open subsets
of the from $V_{i}=U_{i}\times\mathbb{A}^{1}$ on which $\mathbb{G}_{a}$
acts by translations on the second factor. The factoriality of $\mathbb{A}^{3}$
implies in turn that a geometric quotient $\mathbb{A}^{3}/\mathbb{G}_{a}$
exists as a quasi-affine open subset of $\mathbb{A}^{3}/\!/\mathbb{G}_{a}\simeq\mathbb{A}^{2}$
with at most finite complement, and the equality $\mathbb{A}^{3}/\mathbb{G}_{a}=\mathbb{A}^{3}/\!/\mathbb{G}_{a}$
ultimately follows by comparing Euler characteristics. 

Local triviality in the Zariski topology is actually a built-in property
of proper $\mathbb{G}_{a}$-actions on smooth algebraic varieties
of dimension less than four. Indeed, recall that an action $\mu:\mathbb{G}_{a}\times X\rightarrow X$
on an algebraic variety $X$ is said to be proper if the morphism
$\mu\times\mathrm{pr}_{2}:\mathbb{G}_{a}\times X\rightarrow X\times X$
is proper, in this context in fact a closed immersion since $\mathbb{G}_{a}$
has no nontrivial algebraic subgroup. Being in particular set-theoretically
free, such an action is then locally trivial in the \'etale topology,
i.e., there exists an \'etale covering $U\times\mathbb{G}_{a}\rightarrow X$
of $X$ which is equivariant for the action of $\mathbb{G}_{a}$ on
$U\times\mathbb{G}_{a}$ by translations on the second factor. This
implies that a geometric quotient exists in the category of algebraic
spaces in the form of an \'etale locally trivial $\mathbb{G}_{a}$-bundle
$\rho:X\rightarrow X/\mathbb{G}_{a}$ over a certain algebraic space
$X/\mathbb{G}_{a}$, the properness of $\mu$ being then equivalent
to the separatedness of $X/\mathbb{G}_{a}$ (see e.g. \cite{Popp1977}).
Now if $X$ is smooth of dimension at most three, then $X/\mathbb{G}_{a}$
is a smooth separated algebraic space of dimension at most two whence
a quasi-projective variety by virtue of Chow's Lemma. Since $\mathbb{G}_{a}$
is a special group, the $\mathbb{G}_{a}$-bundle $\rho:X\rightarrow X/\mathbb{G}_{a}$
is then in fact locally trivial in the Zariski topology on $X/\mathbb{G}_{a}$
which yields the Zariski local triviality of the $\mathbb{G}_{a}$-action
on $X$. 

For $\mathbb{G}_{a}$-actions on higher dimensional affine spaces,
properness fails in general to imply Zariski local triviality and
Zariski local triviality is no longer sufficient to guarantee that
a proper $\mathbb{G}_{a}$-action is a translation. In particular,
starting from dimension $5$, there exists proper triangular $\mathbb{G}_{a}$-actions
which are not Zariski locally trivial \cite{Deveney1995} and proper
triangular, Zariski locally trivial actions with strictly quasi-affine
geometric quotients \cite{Winkelmann1990}. But the question whether
a proper $\mathbb{G}_{a}$-action on $\mathbb{A}^{4}$ is a translation
or at least Zariski locally trivial remains open and very little progress
has been made in the study of these actions during the last decade.
The only existing partial results so far concern triangular $\mathbb{G}_{a}$-actions
: Deveney, van Rossum and the second author \cite{Deveney2004} established
in 2004 that a Zariski locally trivial triangular $\mathbb{G}_{a}$-action
on $\mathbb{A}^{4}$ is in fact a translation. The proof depends on
the very particular structure of the ring of invariants for such actions
and hence cannot be adapted to more general actions. The second positive
result concerns a special type of triangular $\mathbb{G}_{a}$-actions
called \emph{twin-triangular,} corresponding to locally nilpotent
derivations of $\mathbb{C}[x,y,z_{1},z_{2}]$ of the form $\partial=r(x)\partial_{y}+p_{1}(x,y)\partial_{z_{1}}+p_{2}(x,y)\partial z_{2}$
where $r(x)\in\mathbb{C}\left[x\right]$ and $p_{1}(x,y),p_{2}(x,y)\in\mathbb{C}[x,y]$.
It was established by Deveney and the second author \cite{Deveney2002}
that a proper twin-triangular $\mathbb{G}_{a}$-action corresponding
to a derivation for which the polynomial $r(x)$ has simple roots
is a translation. This was accomplished by explicitly computing the
invariant ring $\mathbb{C}[x,y,z_{1},z_{2}]^{\mathbb{G}_{a}}$ and
investigating the structure of the algebraic quotient morphism $\mathbb{A}^{4}\rightarrow\mathbb{A}^{4}/\!/\mathbb{G}_{a}=\mathrm{Spec}(\mathbb{C}[x,y,z_{1},z_{2}]^{\mathbb{G}_{a}})$.
While a result of Daigle and Freudenburg \cite{Daigle1998} gives
finite generation of $\mathbb{C}[x,y,z_{1},z_{2}]^{\mathbb{G}_{a}}$
for arbitrary triangular $\mathbb{G}_{a}$-actions, there is no a
priori bound on the number of its generators, and the simplicity of
the roots of $r(x)$ was crucial to achieve the computation of these
rings.

Here we consider the more general case of twin-triangular actions
of $\mathbb{G}_{a}=\mathbb{G}_{a,X}=\mathbb{G}_{a,\mathbb{C}}\times_{{\rm Spec}(\mathbb{C})}X$
on an affine space $\mathbb{A}_{X}^{3}$ over the spectrum $X$ of
a complex Dedekind domain $A$. Removing in particular the condition
on simplicity of the roots of $r$, we show that a proper $\mathbb{G}_{a}$-action
on $\mathbb{A}_{X}^{3}$ generated by an $A$-derivation of $A[y,z_{1},z_{2}]$
of the form $\partial=r\partial_{y}+p_{1}(y)\partial_{z_{1}}+p_{2}(y)\partial_{z_{2}}$,
$r\in A$, $p_{1},p_{2}\in A[y]$ is a translation, i.e. the geometric
quotient $\mathbb{A}_{X}^{3}/\mathbb{G}_{a}$ is $X$-isomorphic to
$\mathbb{A}_{X}^{2}$ and $\mathbb{A}_{X}^{3}$ is equivariantly isomorphic
to $\mathbb{A}_{X}^{3}/\mathbb{G}_{a}\times_{X}\mathbb{G}_{a}$ where
$\mathbb{G}_{a}$ acts by translations on the second factor. Even
though finite generation of the rings of invariant for triangular
$A$-derivations of $A[y,z_{1},z_{2}]$ holds in this more general
setting thanks to the aforementioned result of Daigle and Freudenburg,
our approach avoids the computation of these rings and focuses more
on the nature of the geometric quotients $\mathbb{A}_{X}^{3}/\mathbb{G}_{a}$.
As noted before, these quotients a priori exist only as separated
algebraic spaces and the crucial step is to show that for the actions
under consideration they are in fact schemes, or, equivalently that
proper twin-triangular $\mathbb{G}_{a}$-actions on $\mathbb{A}_{X}^{3}$
are not only locally trivial in the \'etale topology but also in
the Zariski topology. Indeed, if so then a straightforward generalization
of the aforementioned result of Deveney, van Rossum and the second
author shows that such Zariski locally trivial triangular $\mathbb{G}_{a}$-actions
are in fact translations. 

To explain the main idea of our proof, let us assume for simplicity
that $A=\mathbb{C}\left[x\right]_{(x)}$ and consider a triangular
$A$-derivation $\partial=x^{n}\partial_{y}+p_{1}(y)\partial_{z_{1}}+p_{2}(y,z_{1})\partial_{z_{2}}$
of $A[y,z_{1},z_{2}]$ generating a proper action on $\mathbb{A}_{X}^{3}$
that we denote by $\mathbb{G}_{a,\partial}$. Being triangular, the
action of $\mathbb{G}_{a,\partial}$ commutes with that $\mathbb{G}_{a,\partial_{z_{2}}}$
defined by the partial derivative $\partial_{z_{2}}$ and descends
to an action on $\mathbb{A}_{X}^{2}=\mathrm{Spec}(A[y,z_{1}])\simeq\mathbb{A}_{X}^{3}/\mathbb{G}_{a,\partial_{z_{2}}}$
corresponding with that generated by the derivation $x^{n}\partial_{y}+p_{1}(y)\partial_{z_{1}}$.
Similarly, the action of $\mathbb{G}_{a,\partial_{z_{2}}}$ on $\mathbb{A}_{X}^{3}$
descends to the geometric quotient $\mathbb{A}_{X}^{3}/\mathbb{G}_{a,\partial}$
. These induced actions are in general no longer set-theoretically
free but if we take the quotient of $\mathbb{A}_{X}^{2}$ by $\mathbb{G}_{a,\partial}$
as an algebraic stack $[\mathbb{A}_{X}^{2}/\mathbb{G}_{a,\partial}]$
we obtain a cartesian square \[\xymatrix{ \mathbb{A}^3_X \ar[d]_{\mathrm{pr}_{y,z_1}} \ar[r] & \mathbb{A}^3_X/\mathbb{G}_{a,\partial} \ar[d] \\ \mathbb{A}^2_X \ar[r] & [\mathbb{A}^2_X/\mathbb{G}_{a,\partial}] }\]
which identifies $[\mathbb{A}_{X}^{2}/\mathbb{G}_{a,\partial}]$ with
the algebraic stack quotient $[(\mathbb{A}_{X}^{3}/\mathbb{G}_{a,\partial})/\mathbb{G}_{a,\partial_{z_{2}}}]$.
In this setting, the Zariski local triviality of a proper triangular
$\mathbb{G}_{a}$-action on $\mathbb{A}_{X}^{3}$ becomes equivalent
to the statement that a separated algebraic $X$-space $V$ admitting
a $\mathbb{G}_{a}$-action with algebraic stack quotient $[V/\mathbb{G}_{a}]$
isomorphic to that of a triangular $\mathbb{G}_{a}$-action on $\mathbb{A}_{X}^{2}$
is in fact a scheme. While a direct proof (or disproof) of this equivalent
characterization seems totally out of reach with existing methods,
we establish that it holds at least over suitable $\mathbb{G}_{a,\partial}$-invariant
principal open subsets $U_{1}$ of $\mathbb{A}_{X}^{2}=\mathrm{Spec}(A[y,z_{1}])$
faithfully flat over $X$ and whose algebraic stack quotients $[U_{1}/\mathbb{G}_{a,\partial}]$
are in fact represented by locally separated algebraic spaces $U_{1}/\mathbb{G}_{a,\partial}$.
So this provides at least a $\mathbb{G}_{a,\partial}$-invariant principal
open subset $V_{1}=\mathrm{pr}_{x,z_{1}}^{-1}(U_{1})\simeq U_{1}\times\mathrm{Spec}(\mathbb{C}[z_{2}])$
of $\mathbb{A}_{X}^{3}$, faithfully flat over $X$, and for which
the Zariski open sub-space $V_{1}/\mathbb{G}_{a,\partial}$ of $\mathbb{A}_{X}^{3}/\mathbb{G}_{a,\partial}$
is a scheme. 

This is where twin-triangularity enters the argument: indeed for such
actions the symmetry between the variables $z_{1}$ and $z_{2}$ enables
the same construction with respect to the other projection $\mathrm{pr}_{y,z_{2}}:\mathbb{A}_{X}^{3}\rightarrow\mathbb{A}_{X}^{2}=\mathrm{Spec}(A[y,z_{2}])$
providing a second Zariski open sub-scheme $V_{2}/\mathbb{G}_{a,\partial}$
of $\mathbb{A}_{X}^{3}/\mathbb{G}_{a,\partial}$ faithfully flat over
$X$. Since the action of $\mathbb{G}_{a,\partial}$ is by definition
equivariantly trivial over the complement of the closed point $0$
of $X$, its local triviality in the Zariski topology follows provided
that the invariant affine open subsets $V_{1}$ and $V_{2}$ can be
chosen so that their union covers the closed fiber of $\mathrm{pr}_{X}:\mathbb{A}_{X}^{3}\rightarrow X$.
\\

With this general strategy in mind, the scheme of the proof is fairly
streamlined. In the first section, we describe algebraic spaces that
arise as geometric quotients of certain affine open subsets $U$ of
an affine plane $\mathbb{A}_{X}^{2}$ over a Dedekind domain equipped
with a triangular $\mathbb{G}_{a}$-action. Then we establish the
crucial property that for such affine open subsets $U$, a proper
lift to $U\times\mathbb{A}^{1}$ of the induced $\mathbb{G}_{a}$-action
on $U$ is equivariantly trivial with affine geometric quotient. This
criterion is applied in the second section to deduce that proper twin-triangular
$\mathbb{G}_{a}$-actions on an affine $3$-space $\mathbb{A}_{X}^{3}$
over a complex Dedekind domain are locally trivial in the Zariski
topology.

\section{Preliminaries on triangular $\mathbb{G}_{a}$-actions on an affine
plane over a Dedekind domain }

This section is devoted to the study of certain algebraic spaces that
arise as geometric quotients for triangular $\mathbb{G}_{a}$-actions
on suitably chosen invariant open subsets in $\mathbb{A}_{X}^{2}$. 

\begin{parn} As a motivation for what follows, consider a $\mathbb{G}_{a}$-action
on $\mathbb{A}^{3}=\mathbb{A}^{1}\times\mathbb{A}^{2}={\rm Spec}(\mathbb{C}[x][y,z])$
generated by a triangular derivation $\partial=x^{n}\partial_{y}+p(y)\partial_{z}$
of $\mathbb{C}[x,y,z]$, where $n\geq1$ and where $p(y)\in\mathbb{C}[y]$
is a non constant polynomial. Letting $P(y)\in\mathbb{C}[y]$ be an
integral of $p$, the polynomials $x$ and $t=-x^{n}z+P(y)$ generate
the algebra of invariants $\mathbb{C}[x,y,z]^{\mathbb{G}_{a}}={\rm Ker}\partial$.
Corresponding to the fact that $y/x^{n}$ is a slice for $\partial$
on the principal invariant open subset $\{x\neq0\}$ of $\mathbb{A}^{3}$,
the quotient morphism $q:\mathbb{A}^{3}\rightarrow\mathbb{A}^{3}/\!/\mathbb{G}_{a}={\rm Spec}\left(\mathbb{C}\left[x\right][t]\right)$
restricts to a trivial principal $\mathbb{G}_{a}$-bundle over the
open subset $\left\{ x\neq0\right\} $ of $\mathbb{A}^{3}/\!/\mathbb{G}_{a}$.
In contrast, the set-theoretic fiber of $q$ over a point $(0,t_{0})\in\mathbb{A}^{3}/\!/\mathbb{G}_{a}$
consists of a disjoint union of affine lines in bijection with the
roots of $P(y)-t_{0}$, each simple root corresponding in particular
to an orbit of the action. Thus $\mathbb{A}^{3}/\!/\mathbb{G}_{a}$
is in general far from being even a set-theoretic orbit space for
the action. However, the observation that the inverse image by $q$
of the line $L_{0}=\left\{ x=0\right\} \subset\mathbb{A}^{3}/\!/\mathbb{G}_{a}$
is equivariantly isomorphic to the product $L_{1}\times\mathbb{A}^{1}={\rm Spec}(\mathbb{C}[y][z])$
on which $\mathbb{G}_{a}$ acts via the twisted translation generated
by the derivation $p(y)\partial_{z}$ of $\mathbb{C}[y,z]$ suggests
that a better geometric object approximating an orbit space for the
action should be obtained from $\mathbb{A}^{3}/\!/\mathbb{G}_{a}$
by replacing the line $L_{0}$ by $L_{1}$ , considered as total space
of the finite cover $h_{0}:L_{1}\rightarrow L_{0}$, $y\mapsto t=P(y)$.

On the other hand, on every invariant open subset $V$ of $\mathbb{A}^{3}$
on which the action restricts to a set-theoretically free $\mathbb{G}_{a}$-action,
a geometric quotient $\rho:V\rightarrow V/\mathbb{G}_{a}$ exists
in the form an \'etale locally trivial $\mathbb{G}_{a}$-bundle over
an algebraic space $V/\mathbb{G}_{a}$. By definition of $\partial$,
the fixed points of the $\mathbb{G}_{a}$-action are supported on
the disjoint union of lines $\left\{ x=p(y)=0\right\} $. Therefore,
letting $C_{0}\subset L_{0}={\rm Spec}(\mathbb{C}[t])$ be the complement
of the branch locus of $h_{0}$ and considering $\mathbb{A}^{1}\times C_{0}$
as an open subset of $\mathbb{A}^{3}/\!/\mathbb{G}_{a}$, a geometric
quotient exists on the open subset $V=q^{-1}(\mathbb{A}^{1}\times C_{0})$
of $\mathbb{A}^{3}$. In view of the previous discussion, the algebraic
quotient morphism $q\mid_{V}:V\rightarrow V/\!/\mathbb{G}_{a}\simeq\mathbb{A}^{1}\times C_{0}\subset\mathbb{A}^{3}/\!/\mathbb{G}_{a}$
should thus factor through a $\mathbb{G}_{a}$-bundle $\rho:V\rightarrow V/\mathbb{G}_{a}$
over an algebraic space $V/\mathbb{G}_{a}$ obtained from $\mathbb{A}^{1}\times C_{0}$
by replacing the curve $\left\{ 0\right\} \times C_{0}\simeq C_{0}$
by the finite \'etale cover $h_{0}:C_{1}=h_{0}^{-1}(C_{0})\rightarrow C_{0}$
of itself. 

\end{parn}

In what follows, to give precise sense to the above intuitive interpretation,
we review the construction of a particular type of algebraic space
$\mathfrak{S}$ obtained from a surface by ``replacing a curve by
a finite \'etale cover of itself'' and we check that these spaces
do indeed arise as geometric quotients for $\mathbb{G}_{a}$-actions
on certain affine threefolds. Then, conversely, we characterize effectively
which \'etale locally trivial $\mathbb{G}_{a}$-bundles $\rho:V\rightarrow\mathfrak{S}$
over such spaces have an affine total space.

\subsection{Algebraic space surfaces with an irreducible $r$-fold curve }

\indent\newline\noindent Given a smooth affine curve $X={\rm Spec}(A)$,
a closed point $o\in X$ and a finite \'etale morphism $h_{0}:C_{1}={\rm Spec}(R_{1})\rightarrow C_{0}={\rm Spec}(R_{0})$
between smooth connected affine curves, our aim is to construct an
algebraic space $\mathfrak{S}=\mathfrak{S}(X,o,h_{0})$ which looks
like $X\times C_{0}$ but with the special curve $\left\{ o\right\} \times C_{0}\simeq C_{0}$
replaced by $C_{1}$. To obtain such an $\mathfrak{S}$, one can simply
define it as the quotient of $X\times C_{1}$ by the \'etale equivalence
relation $(x,c_{1})\sim(x',c_{1}')\Leftrightarrow(x=x'\neq0\textrm{ and }h_{0}(c_{1})=h_{0}(c_{1}'))$.
More formally, letting $X_{*}=X\setminus\left\{ o\right\} $, this
means that $\mathfrak{S}=X\times C_{1}/R$ where 
\[
\mathrm{diag}\sqcup j:R=X\times C_{1}\sqcup(X\times C_{1})\times_{X_{*}\times C_{0}}(X\times C_{1})\setminus\mathrm{Diag}\longrightarrow(X\times C_{1})\times(X\times C_{1})
\]
is the \'etale equivalence relation defined by the diagonal embedding
$\mathrm{diag}:X\times C_{1}\rightarrow(X\times C_{1})\times(X\times C_{1})$
and the natural immersion $j:(X\times C_{1})\times_{X_{*}\times C_{0}}(X\times C_{1})\setminus\mathrm{Diag}\rightarrow(X\times C_{1})\times(X\times C_{1})$
respectively. This equivalence relation restricts on the invariant
open subset $X_{*}\times C_{1}$ to that defined by the diagonal embedding
$X_{*}\times C_{1}\rightarrow(X_{*}\times C_{1})\times_{X_{*}\times C_{0}}(X_{*}\times C_{1})$
which has quotient $X\times C_{0}$. This implies that the $R$-invariant
morphism $\mathrm{pr}_{1}\times h_{0}:X\times C_{1}\rightarrow X\times C_{0}$
descends to a morphism $\overline{\varphi}:\mathfrak{S}\rightarrow X\times C_{0}$
restricting to an isomorphism over $X_{*}\times C_{0}$. In contrast,
since $R$ induces the trivial equivalence relation on $\{o\}\times C_{1}\simeq C_{1}$,
one has $\overline{\varphi}^{-1}(\{o\}\times C_{0})\simeq C_{1}$
as desired.

A disadvantage of this simple presentation of $\mathfrak{S}$ is that
the equivalence relation $R$ is quasi-finite but not finite. To construct
an alternative presentation of $\mathfrak{S}$ as a quotient of a
suitable scheme $Z$ by a finite \'etale equivalence relation, in
fact by a free action of a finite group $G$, we proceed as follows: 

\begin{parn} \label{par:alg-space-const} We let $C={\rm Spec}(R)$
be the normalization of $C_{0}$ in the Galois closure of the field
extension ${\rm Frac}(R_{0})\hookrightarrow{\rm Frac}(R_{1})$. By
construction, the induced morphism $h:C\rightarrow C_{0}$ is a torsor
under the corresponding Galois group $G$ which factors as $h:C\stackrel{h_{1}}{\rightarrow}C_{1}\stackrel{h_{0}}{\rightarrow}C_{0}$
where $h_{1}:C\rightarrow C_{1}$ is a torsor under a certain subgroup
$H$ of $G$ with index equal to the degree $r$ of the finite morphism
$h_{0}$. Now we let $Z$ be the scheme obtained by gluing $r$ copies
$Z_{\overline{g}}$, $\overline{g}\in G/H$, of $X\times C$ by the
identity outside the curves $\{o\}\times C\subset Z_{\overline{g}}$.
The group $G$ acts freely on $Z$ by $Z_{\overline{g}}\ni(x,t)\mapsto g'\cdot(x,t)=(x,g'\cdot t)\in Z_{\overline{g'\cdot g}}$
and so a geometric quotient $\xi:Z\rightarrow\mathfrak{S}=Z/G$ exists
in the category of algebraic spaces in the form of an \'etale $G$-torsor
over an algebraic space $\mathfrak{S}$. The local morphisms ${\rm pr}_{1}\times h:Z_{\overline{g}}\simeq X\times C\rightarrow X\times C_{0}$,
$\overline{g}\in G/H$, glue to a global $G$-invariant morphism $\varphi:Z\rightarrow X\times C_{0}$
which descends in turn to a morphism $\overline{\varphi}:\mathfrak{S}=Z/G\rightarrow X\times C_{0}$
restricting to an isomorphism outside $\{o\}\times C_{0}$. In contrast,
$\overline{\varphi}^{-1}(\{o\}\times C_{0})$ is isomorphic as a scheme
over $C_{0}$ to the quotient of $C\times(G/H)$ by the diagonal action
of $G$ whence to $C/H\simeq C_{1}$. 

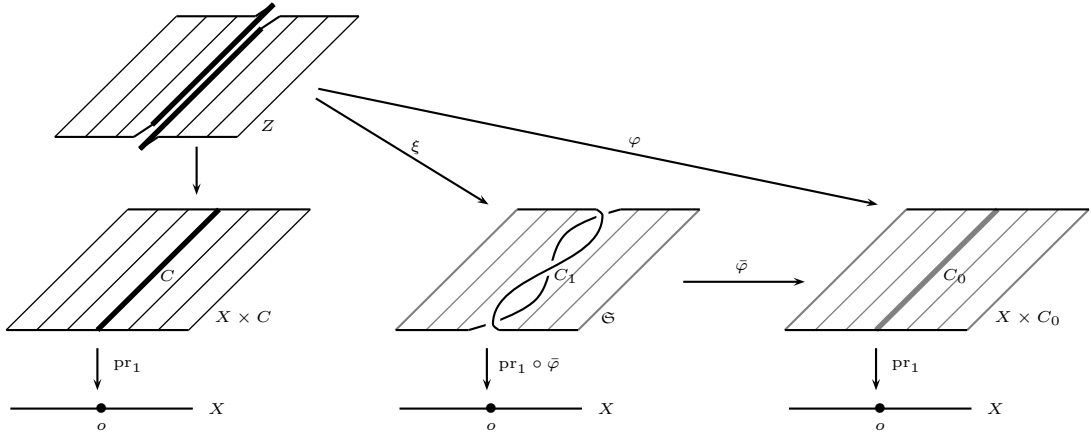
\begin{figure}[ht]
\psset{unit=0.8}
\begin{pspicture}(12,-3)(10,8)
\rput(1,4){\usebox{\XZ}}
\psline{->}(4,3.8)(4,2.8)
\rput(0,0){\usebox{\XC}}
\rput(8,0){\usebox{\Xspace}}
\rput(16,0){\usebox{\XCzero}}
%%\pscurve{->}(6.5,5)(8,5)(20,2.6)
\psline{->}(6.45,4.8)(10,2.6)
\rput(8.5,3.8){{\scriptsize $\xi$}}
\psline{->}(6.5,5)(18,2.6)
\rput(13,3.9){{\scriptsize $\varphi$}}
\psline{->}(14,1)(16.5,1)
\rput(15.2,1.3){{\scriptsize $\bar{\varphi}$}}
\end{pspicture}
\caption{Construction of $\mathfrak{S}$ as a quotient of $Z$ by a finite group action}
\end{figure}

\noindent The fact that the algebraic spaces $\mathfrak{S}=Z/G$
obtained by this construction coincide with the $X\times C_{1}/R$
defined above can be seen as follows. By construction, every open
subset $Z_{\overline{g}}\simeq X\times C$ of $Z$, $\overline{g}\in G/H$,
is invariant under the induced action of $H$, with quotient $Z_{\overline{g}}/H\simeq X\times C/H=X\times C_{1}$.
So the morphism $X\times C\rightarrow\mathfrak{S}$ induced by restricting
$\xi:Z\rightarrow\mathfrak{S}$ to any open subset $Z_{\overline{g}}\subset Z$
descends to an \'etale morphism $X\times C_{1}=X\times C/H\rightarrow\mathfrak{S}$,
and one checks that the \'etale equivalence relation $(\mathrm{pr}_{1},\mathrm{pr}_{2}):(X\times C_{1})\times_{\mathfrak{S}}(X\times C_{1})\rightrightarrows X\times C_{1}$
is precisely isomorphic to that $R\rightrightarrows X\times C_{1}$
defined above. 

\end{parn}
\begin{rem}
Note that if $h_{0}:C_{1}\rightarrow C_{0}$ is a not an isomorphism
then $\mathfrak{S}$ cannot be a scheme. Indeed, otherwise the image
by $\xi$ of a point $z_{0}\in\left\{ o\right\} \times C\subset Z_{\overline{g}}\subset Z$
for some $\overline{g}\in G/H$ would have a Zariski open affine neighborhood
$U$ in $\mathfrak{S}_{h_{0}}$. But then since $\xi:Z\rightarrow\mathfrak{S}$
is a finite morphism, $\xi^{-1}(U)$ would be a $G$-invariant affine
open neighborhood of $z_{0}$ in $Z$, which is absurd as such a point
does not even have a separated open neighborhood in $Z$. 
\end{rem}

\subsection{Geometric quotients for restricted triangular $\mathbb{G}_{a}$-actions
on a relative affine plane}

\indent\newline\noindent Here we show that the algebraic spaces $\mathfrak{S}=\mathfrak{S}(X,o,h_{0})$
described in the previous subsection naturally arise as geometric
quotients for $\mathbb{G}_{a}$-actions on certain open subsets of
affine planes over discrete valuation rings. 

\begin{parn} \label{par:Open-subs-A3} We let $X={\rm Spec}(A)$
be the spectrum of a discrete valuation ring with uniformizing parameter
$x$ and with residue field $\mathbb{C}$. We denote by $o$ its closed
point and we let $\mathbb{A}_{X}^{2}={\rm Spec}(A\left[y,z\right])$.
Given an irreducible triangular locally nilpotent $A$-derivation
$\partial=x^{n}\partial_{y}+p\left(y\right)\partial_{z}$ of $A\left[y,z\right]$,
where $p(y)\in A\left[y\right]$, we let $P\left(y\right)\in A\left[y\right]$
be a integral of $p\left(y\right)$. Since $\partial$ is irreducible,
$p\left(y\right)$ is not divisible by $x$ and so the restriction
$\overline{P}$ of the morphism $P:\mathbb{A}_{X}^{1}={\rm Spec}(A[y])\rightarrow\mathbb{A}_{X}^{1}={\rm Spec}(A[t])$
over the closed point of $X$ is not constant. Its branch locus is
a principal divisor ${\rm div}\left(\alpha\right)$ for a certain
$\alpha\in\mathbb{C}\left[t\right]$ and we let $C_{\partial}={\rm Spec}(R_{0})$,
where $R_{0}=\mathbb{C}\left[t\right]_{\alpha}$, be its complement.
The polynomial $-x^{n}z+P\left(y\right)\in A\left[y,z\right]$ defines
a $\mathbb{G}_{a}$-invariant $X$-morphism $f:\mathbb{A}_{X}^{2}={\rm Spec}(A\left[y,z\right])\rightarrow{\rm Spec}\left(A\left[t\right]\right)$,
smooth over $X\times C_{\partial}$, and such that the induced $\mathbb{G}_{a}$-action
on $V_{\partial}=f^{-1}\left(X\times C_{\partial}\right)\subset\mathbb{A}_{X}^{2}$
is set-theoretically free. Thus a geometric quotient exists in the
category of algebraic spaces in the form of an \'etale locally trivial
$\mathbb{G}_{a}$-bundle $\rho:V_{\partial}\rightarrow V_{\partial}/\mathbb{G}_{a}$.
Clearly, the curve $C_{1}={\rm Spec}(R_{0}\left[y\right]/\left(P(0,y)-t\right))$
is smooth and irreducible, and the induced morphism $h_{0}:C_{1}\rightarrow C_{\partial}$
is finite and \'etale. With the notation of $\S$ \ref{par:alg-space-const}
above, we have the following result:

\end{parn} 
\begin{prop}
\label{prop:Restricted-quotient}The algebraic space quotient $V_{\partial}/\mathbb{G}_{a}$
is isomorphic to $\mathfrak{S}(X,o,h_{0})$.\end{prop}
\begin{proof}
Again, we let $h:C={\rm Spec}(R)\rightarrow C_{\partial}$ be the
Galois closure of the finite \'etale morphism $h_{0}:C_{1}\rightarrow C_{\partial}$.
By construction, the polynomial $\overline{P}(y)-t\in R\left[y\right]$
splits as $\overline{P}(y)-t=\prod_{\overline{g}\in G/H}(y-t_{\overline{g}})$
for certain elements $t_{\overline{g}}\in R$, $\overline{g}\in G/H$,
on which the Galois group $G$ acts by permutation. Furthermore, since
$h_{0}:C_{1}\rightarrow C_{\partial}$ is \'etale, it follows that
for distinct $\overline{g},\overline{g}'\in G/H$, one has $t_{\overline{g}}(c)\neq t_{\overline{g'}}(c)$
for every point $c\in C$. Now a similar argument as in the proof
of Theorem 3.2 in \cite{Dubouloz2009} implies that there exists a
collection of elements $\sigma_{\overline{g}}\in A\otimes_{\mathbb{C}}R$
with respective residue classes $t_{\overline{g}}\in R$ modulo $x$,
$\overline{g}\in G/H$, on which $G$ acts by permutation, a polynomial
$S_{1}\in A\otimes_{\mathbb{C}}R\left[y\right]$ with invertible residue
class modulo $x$ and a polynomial $S_{2}\in A\otimes_{\mathbb{C}}R\left[y\right]$
such that in $A\otimes_{\mathbb{C}}R\left[y\right]$ one can write
\[
P(y)-t=S_{1}(y)\prod_{\overline{g}\in G/H}(y-\sigma_{\overline{g}})+x^{n}S_{2}(y).
\]
This implies that\emph{ $W=V_{\partial}\times_{C_{\partial}}C\simeq{\rm Spec}\left(A\otimes_{\mathbb{C}}R\left[y,z\right]/(x^{n}z-P(y)+t\right)$
}is isomorphic to the sub-variety of $C\times\mathbb{A}_{X}^{2}$
defined by the equation $x^{n}z=\tilde{P}\left(y\right)=S_{1}(y)\prod_{\overline{g}\in G/H}(y-\sigma_{\overline{g}})$.\emph{
}Furthermore, the $\mathbb{G}_{a}$-action of $V_{\partial}$ lifts
to the set-theoretically free $\mathbb{G}_{a}$-action on $W$ commuting
with that of $G$ associated with the locally nilpotent $A\otimes_{\mathbb{C}}R$-derivation
$x^{n}\partial_{y}+\partial_{y}(\tilde{P}(y))\partial_{z}$. Then
a standard argument (see e.g. \emph{loc. cit.} or \cite{Dubouloz2011b})
shows that the $\mathbb{G}_{a}$-invariant morphism ${\rm pr}_{X,C}:W\rightarrow X\times C$
factors through a $G$-equivariant $\mathbb{G}_{a}$-bundle $\eta:W\rightarrow Z$
over the scheme $Z$ as in $\S$ \ref{par:alg-space-const} above with
local trivializations $W\mid_{Z_{\overline{g}}}\simeq Z_{\overline{g}}\times{\rm Spec}(\mathbb{C}[u_{\overline{g}}])$,
where $u_{\overline{g}}=x^{-n}(y-\sigma_{\overline{g}})$, $\overline{g}\in G/H$,
and transition isomorphisms over $Z_{\overline{g}}\cap Z_{\overline{g}'}\simeq{\rm Spec}(A_{x}\otimes_{\mathbb{C}}R)$
of the form $u_{\overline{g}}\mapsto u_{\overline{g}'}=u_{\overline{g}}+x^{-n}(\sigma_{\overline{g}}-\sigma_{\overline{g}'})$
for every pair of distinct elements $\overline{g},\overline{g}'\in G/H$.
By construction, we have a cartesian square \[\xymatrix{ W \ar[r] \ar[d]_{\eta} & V_\partial \simeq V/G \ar[d]^{\rho} \\ Z \ar[r] & \mathfrak{S}=Z/G,}\]
where the horizontal arrows are $G$-torsors and the vertical ones
are $\mathbb{G}_{a}$-bundles, which provides, by virtue of the universal
property of categorical quotients, an isomorphism of algebraic spaces
$V_{\partial}/\mathbb{G}_{a}\simeq\mathfrak{S}=\mathfrak{S}(X,o,h_{0})$.
\end{proof}

\subsection{Criteria for affineness}

\indent\newline\noindent Even though Proposition \ref{prop:Restricted-quotient}
shows in particular that algebraic spaces of the form $\mathfrak{S}=\mathfrak{S}(X,o,h_{0})$
may arise as geometric quotient for $\mathbb{G}_{a}$-actions on affine
schemes, the total space of an \'etale locally trivial $\mathbb{G}_{a}$-bundle
$\rho:V\rightarrow\mathfrak{S}$ is in general neither a scheme nor
even a separated algebraic space. However it is possible to characterize
effectively which $\mathbb{G}_{a}$-bundles $\rho:V\rightarrow\mathfrak{S}$
have affine total space. 

\begin{parn} Indeed, with the notation of $\S$ \ref{par:alg-space-const}
above, since $X\times C_{0}$ is affine, the affineness of $V$ is
equivalent to that of the morphism $\overline{\varphi}\circ\rho:V\rightarrow X\times C_{0}$.
Furthermore, since $\rho:V\rightarrow\mathfrak{S}$ is an affine morphism
and $\overline{\varphi}:\mathfrak{S}\rightarrow X\times C_{0}$ is
an isomorphism outside the curve $\{o\}\times C_{0}$, it is enough
to consider the case that $X={\rm Spec}\left(A\right)$ is the spectrum
of a discrete valuation ring with closed point $o$ and uniformizing
parameter $x$. Every $\mathbb{G}_{a}$-bundle $\rho:V\rightarrow\mathfrak{S}$
pulls-back via the Galois cover $\xi:Z\rightarrow\mathfrak{S}=Z/G$
to a $G$-equivariant $\mathbb{G}_{a}$-bundle $\eta={\rm pr}_{2}:W=V\times_{\mathfrak{S}}Z\rightarrow Z$.
By construction of $Z$, the latter becomes trivial on the canonical
covering $\mathcal{U}$ of $Z$ by the affine open subsets $Z_{\overline{g}}\simeq X\times C$,
$\overline{g}\in G/H$, whence is determined up to isomorphism by
a $G$-equivariant \v{C}ech $1$-cocyle 
\[
\{f_{\overline{g}\overline{g}'}\}\in C^{1}(\mathcal{U},\mathcal{O}_{Z})\simeq\bigoplus_{\overline{g},\overline{g}'\in G/H,\overline{g}\neq\overline{g}'}A_{x}\otimes_{\mathbb{C}}R.
\]
With this notation, we have the following criterion: 

\end{parn}
\begin{thm}
\label{thm:Aff-criterion} For a $\mathbb{G}_{a}$-bundle $\rho:V\rightarrow\mathfrak{S}$,
the following are equivalent:

a) $V$ is a separated algebraic space,

b) For every every pair of distinct elements $\overline{g},\overline{g}'\in G/H$,
there exists an element $\tilde{f}_{\overline{g}\overline{g}'}\in A\otimes_{\mathbb{C}}R$
with invertible residue class modulo $x$ such that $f_{\overline{g}\overline{g}'}=x^{-l}\tilde{f}_{\overline{g}\overline{g}'}$
for a certain $l>1$. 

c) $V$ is an affine scheme.\end{thm}
\begin{proof}
By virtue of \cite[Proposition 10.1.2 and Lemma 10.1.3 ]{Dubouloz2004a},
b) is equivalent to the separatedness of the total space of the $\mathbb{G}_{a}$-bundle
$\eta:W\rightarrow Z$ and this property is also equivalent to the
affineness of $W$ thanks to the generalization of the so-called Fieseler
criterion for affineness \cite{Fieseler1994} established in \cite[Theorem 10.2.1]{Dubouloz2004a}.
Now if $V$ is a separated algebraic space then so is $W=V\times_{\mathfrak{S}}Z$
as the projection ${\rm pr}_{1}:W\rightarrow V$ is a $G$-torsor
whence a proper morphism. Thus $W$ is in fact an affine scheme and
so $V\simeq W/G\simeq{\rm Spec}(\Gamma(W,\mathcal{O}_{W})^{G})$ is
an affine scheme as well.
\end{proof}
\begin{parn} Given a $\mathbb{G}_{a}$-bundle $\rho:V\rightarrow\mathfrak{S}$
with affine total space $V$, we have a one-to-one correspondence
between $\mathbb{G}_{a}$-bundles over $\mathfrak{S}$ and lifts of
the $\mathbb{G}_{a}$-action on $V$ to $V\times\mathbb{A}^{1}$.
Indeed, if $\rho':V'\rightarrow\mathfrak{S}$ is another $\mathbb{G}_{a}$-bundle
then the fiber product $V'\times_{\mathfrak{S}}V$ is a $\mathbb{G}_{a}$-bundle
over $V$ via the second projection, whence is isomorphic to the trivial
one $V\times\mathbb{A}^{1}$ on which $\mathbb{G}_{a}$ acts by translation
on the second factor. Via this isomorphism, the natural lift to $V'\times_{\mathfrak{S}}V$
of the $\mathbb{G}_{a}$-action on $V$ defined by $t\cdot\left(v',v\right)=(v',t\cdot v)$
coincides with a lift of it to $V\times\mathbb{A}^{1}$ with geometric
quotient $V\times\mathbb{A}^{1}/\mathbb{G}_{a}\simeq V'$. Conversely,
since every lift to $V\times\mathbb{A}^{1}$ of the $\mathbb{G}_{a}$-action
on $V$ commutes with that by translations on the second factor, the
equivariant projection ${\rm pr}_{1}:V\times\mathbb{A}^{1}\rightarrow V$
descends to a $\mathbb{G}_{a}$-bundle $\rho':V'=V\times\mathbb{A}^{1}/\mathbb{G}_{a}\rightarrow\mathfrak{S}=V/\mathbb{G}_{a}$
fitting into a cartesian square \[\xymatrix{V\times \mathbb{A}^1 \ar[d]_{{\rm pr}_1} \ar[r] & V'=V\times \mathbb{A}^1/\mathbb{G}_a \ar[d]^{\rho'} \\ V \ar[r]^{\rho} & \mathfrak{S}=V/\mathbb{G}_a} \] of
$\mathbb{G}_{a}$-bundles. In this diagram the horizontal arrows correspond
to the $\mathbb{G}_{a}$-actions on $V$ and its lift to $V\times\mathbb{A}^{1}$
while the vertical ones correspond to the actions on $V\times\mathbb{A}^{1}$
by translations on the second factor and the one it induces on $V\times\mathbb{A}^{1}/\mathbb{G}_{a}$.
Combined with Theorem \ref{thm:Aff-criterion}, this correspondence
leads to the following criterion:

\end{parn}
\begin{cor}
\label{cor:Affine-extended-quotient} Let $\rho:V\rightarrow\mathfrak{S}$
be a $\mathbb{G}_{a}$-bundle with affine total space over an algebraic
space $\mathfrak{S}$ as in $\S$ \ref{par:alg-space-const}. Then the
total space of a $\mathbb{G}_{a}$-bundle $\rho':V'\rightarrow\mathfrak{S}$
is an affine scheme if and only if the corresponding lifted $\mathbb{G}_{a}$-action
on $V\times\mathbb{A}^{1}$ is proper. \end{cor}
\begin{proof}
Since properness of the lifted $\mathbb{G}_{a}$-action on $V\times\mathbb{A}^{1}$
is equivalent to the separatedness of the algebraic space $V'\simeq V\times\mathbb{A}^{1}/\mathbb{G}_{a}$,
the assertion is a direct consequence of Theorem \ref{thm:Aff-criterion}
above. 
\end{proof}

\section{Twin triangular $\mathbb{G}_{a}$-actions of affine $3$-spaces over
Dedekind domains }

In what follows, we let $X$ be the spectrum of a Dedekind domain
$A$ over $\mathbb{C}$, and we let $\mathbb{A}_{X}^{3}$ be the spectrum
of the polynomial ring $A[y,z_{1},z_{2}]$ in three variables over
$A$. Algebraic actions of $\mathbb{G}_{a,X}=\mathbb{G}_{a}\times_{{\rm Spec}\left(\mathbb{C}\right)}X$
on $\mathbb{A}_{X}^{3}$ are in one-to-one correspondence with locally
nilpotent $A$-derivations of $A[y,z_{1},z_{2}]$. Such an action
is called triangular if the corresponding derivation can be written
as $\partial=r\partial_{y}+p_{1}(y)\partial_{z_{1}}+p_{2}(x,y)\partial z_{2}$
for some $r\in A$, $p_{1}\in A[y]$ and $p_{2}\in A[y,z_{1}]$. A
triangular $\mathbb{G}_{a,X}$-action on $\mathbb{A}_{X}^{3}$ is
said to be \emph{twin-triangular} if the corresponding $p_{2}$ belongs
to the sub-ring $A[y]$ of $A[y,z_{1}]$.

\subsection{Proper twin-triangular $\mathbb{G}_{a}$-actions are translations}

\indent\newline\noindent This sub-section is devoted to the proof
of the following result:
\begin{thm}
\label{thm:Main-Theorem} A proper twin-triangular $\mathbb{G}_{a,X}$-action
on $\mathbb{A}_{X}^{3}$ is a translation, i.e., $\mathbb{A}_{X}^{3}/\mathbb{G}_{a,X}$
is $X$-isomorphic to $\mathbb{A}_{X}^{2}$ and $\mathbb{A}_{X}^{3}$
is equivariantly isomorphic to $\mathbb{A}_{X}^{3}/\mathbb{G}_{a}\times_{X}\mathbb{G}_{a,X}$
where $\mathbb{G}_{a,X}$ acts by translations on the second factor. 
\end{thm}
\begin{parn} The argument of the proof given below can be decomposed
in two steps : we first establish in Proposition \ref{prop:Loc-triv-TR}
that any Zariski locally trivial triangular $\mathbb{G}_{a,X}$-action
on $\mathbb{A}_{X}^{3}$ is a translation. This reduces the problem
to showing that a proper twin-triangular $\mathbb{G}_{a,X}$-action
on $\mathbb{A}_{X}^{3}$ is not only equivariantly trivial in the
\'etale topology, which always holds for a proper whence free $\mathbb{G}_{a,X}$-action,
but also in the Zariski topology. This is done in Proposition \ref{prop:Twin-Loc-trivi}.
In the sequel, unless otherwise specified, we implicitly work in the
category of schemes over $X$ and we denote $\mathbb{G}_{a,X}$ simply
by $\mathbb{G}_{a}$.

\end{parn}

\noindent We begin with the following generalization of Theorem 2.1
in \cite{Deveney2004}:
\begin{prop}
\label{prop:Loc-triv-TR} Let $A$ be a Dedekind domain over $\mathbb{C}$
and let $\partial$ be a triangular $A$-derivation of $A[y,z_{1},z_{2}]$
generating a Zariski locally trivial $\mathbb{G}_{a}$-action on $\mathbb{A}_{X}^{3}={\rm Spec}(A\left[y,z_{1},z_{2}\right])$.
Then the action is equivariantly trivial with quotient isomorphic
to $\mathbb{A}_{X}^{2}$. \end{prop}
\begin{proof}
The hypotheses imply that $\mathbb{A}_{X}^{3}$ has the structure
of Zariski locally trivial $\mathbb{G}_{a}$-bundle over a a quasi-affine
$X$-scheme $\psi:Y=\mathbb{A}_{X}^{3}/\mathbb{G}_{a}\rightarrow X$
(see e.g. \cite{Deveney1994}). Furthermore, since each fiber, closed
or not, of the invariant morphism ${\rm pr}_{X}:\mathbb{A}_{X}^{3}\rightarrow X$
is isomorphic to an affine $3$-space equipped with an induced free
triangular $\mathbb{G}_{a}$-action, it follows from \cite{Snow1988}
that all fibers of $\psi:Y\rightarrow X$ are isomorphic to affine
planes over the corresponding residue fields. It is enough to show
that $Y$ is an affine $X$-scheme. Indeed, if so, then by virtue
of \cite{Sathaye1983}, $\psi:Y\rightarrow X$ is in fact a locally
trivial $\mathbb{A}^{2}$-bundle in the Zariski topology whence a
vector bundle of rank $2$ by \cite{Bass1977}. Furthermore, the affineness
of $Y$ implies that the quotient morphism $\mathbb{A}_{X}^{3}\rightarrow Y$
is a trivial $\mathbb{G}_{a}$-bundle. Thus $Y\times\mathbb{A}^{1}\simeq\mathbb{A}_{X}^{3}$
as bundles over $X$ and so $\psi:Y\rightarrow X$ is the trivial
bundle $\mathbb{A}_{X}^{2}$ over $X$ by virtue of \cite[IV 3.5]{Bass1968}.
The affineness of $\psi:Y\rightarrow X$ being a local question with
respect to the Zariski topology on $X$, we may reduce to the case
where $A$ is a discrete valuation ring with uniformizing parameter
$x$ and residue field $\mathbb{C}$. Since $\Gamma(Y,\mathcal{O}_{Y})\simeq A[y,z_{1},z_{2}]^{\mathbb{G}_{a}}$
is finitely generated by virtue of \cite{Daigle1998}, it is enough
to show that the canonical morphism $\alpha:Y\rightarrow Z={\rm Spec}(A[y,z_{1},z_{2}]^{\mathbb{G}_{a}})$
is surjective, whence an isomorphism. If $\partial y\in A^{*}$ then
the result is clear. Otherwise if $\partial y=0$ then the assertion
follows from \emph{loc. cit.} We may thus assume that $\partial y\in xA\setminus\left\{ 0\right\} $
and then the result follows verbatim from the argument of \cite[Theorem 2.1]{Deveney2004}
which shows that $\alpha$ is surjective over the closed point of
$X$. 
\end{proof}
\noindent Now it remains to show the following:
\begin{prop}
\label{prop:Twin-Loc-trivi} A proper twin-triangular $\mathbb{G}_{a}$-action
on $\mathbb{A}_{X}^{3}$ is locally trivial in the Zariski topology. \end{prop}
\begin{proof}
The question is local in the Zariski topology on $X$. Since the corresponding
derivation $\partial=r\partial_{y}+p_{1}(y)\partial_{z_{1}}+p_{2}(y)\partial z_{2}$
of $A[y,z_{1},z_{2}]$ has a slice over the principal open subset
$D_{r}$ of $X$, whence is equivariantly trivial over it, we may
reduce after localizing at the finitely many maximal ideals of $A$
containing $r$ to the case where $A$ is discrete valuation ring
with uniformizing parameter $x$ and $r=x^{n}$ for some $n\geq1$.
Then it is enough to show that the closed fiber $\mathbb{A}_{o}^{3}$
of the projection ${\rm pr}_{X}:\mathbb{A}_{X}^{3}\rightarrow X$
is contained in a union of invariant open subsets of $\mathbb{A}_{X}^{3}$
on which the induced actions are equivariantly trivial. By virtue
of Lemma \ref{lem:Bad-Plane-Removal} below, we may assume up to a
coordinate change preserving twin-triangularity that the residue classes
$\overline{p}_{i}\in\mathbb{C}[y]$ of the $p_{i}$'s modulo $x$
are non constant and that the inverse images of the branch loci of
the morphisms $\overline{P}_{i}:{\rm Spec}\left(\mathbb{C}\left[y\right]\right)\rightarrow{\rm Spec}\left(\mathbb{C}\left[t\right]\right)$
defined by suitable integrals $\overline{P}_{i}$ of $\overline{p}_{i}$,
$i=1,2$ are disjoint. The first property guarantees that the triangular
derivations $\partial_{i}=x^{n}\partial_{y}+p_{i}(y)\partial_{z_{i}}$
of $A\left[y,z_{i}\right]$, $i=1,2$, are both irreducible. Furthermore,
if we let $V_{\partial_{i}}$ be the invariant open subset of $\mathbb{A}_{X,i}^{2}={\rm Spec}(A\left[y,z_{i}\right])$,
$i=1,2$, equipped with $\mathbb{G}_{a}$-action associated with $\partial_{i}$
as defined in $\S$ \ref{par:Open-subs-A3} above, then the second property
implies that $\mathbb{A}_{o}^{3}$ is contained in the union of the
open subsets ${\rm pr}_{z_{i}}^{-1}(V_{\partial_{i}})\simeq V_{\partial_{i}}\times\mathbb{A}^{1}$,
where ${\rm pr}_{z_{i}}:\mathbb{A}_{X}^{3}\rightarrow\mathbb{A}_{X,i}^{2}$,
$i=1,2$, are the natural projections. These projections being equivariant,
the $\mathbb{G}_{a}$-action on ${\rm \mathbb{A}_{X}^{3}}$ restricts
on ${\rm pr}_{z_{i}}^{-1}(V_{\partial_{i}})\simeq V_{\partial_{i}}\times\mathbb{A}^{1}$
to a proper lift of that on $V_{\partial_{i}}$, $i=1,2$, and so
the geometric quotients ${\rm pr}_{z_{i}}^{-1}(V_{\partial_{i}})/\mathbb{G}_{a}$,
$i=1,2$, are affine schemes by virtue of Corollary \ref{cor:Affine-extended-quotient}.
This implies in turn that the induced actions on the open subsets
${\rm pr}_{z_{i}}^{-1}(V_{\partial_{i}})$, $i=1,2$, are equivariantly
trivial and completes the proof. 
\end{proof}
\noindent In the proof of Proposition \ref{prop:Twin-Loc-trivi},
we exploited the following crucial technical fact concerning set-theoretically
free twin-triangular $\mathbb{G}_{a}$-actions:
\begin{lem}
\label{lem:Bad-Plane-Removal} Let $A$ be a discrete valuation ring
over $\mathbb{C}$ with uniformizing parameter $x$. A twin-triangular
$A$-derivation $\partial$ of $A[y,z_{1},z_{2}]$ generating a set-theoretically
free $\mathbb{G}_{a}$-action is conjugate to a one of the form $x^{n}\partial_{y}+p_{1}(y)\partial_{z_{1}}+p_{2}(y)\partial_{z_{2}}$
with the following properties:

a) The residue classes $\overline{p}_{i}\in\mathbb{C}[y]$ of the
polynomials $p_{i}\in A[y]$ modulo $x$, $i=1,2$, are both non zero
and relatively prime,

b) There exists integrals $\overline{P}_{i}\in\mathbb{C}[y]$ of $\overline{p}_{i}$,
$i=1,2$, for which the inverse images of the branch loci of the morphisms
$\overline{P}_{i}:\mathbb{A}^{1}\rightarrow\mathbb{A}^{1}$, $i=1,2$,
are disjoint. \end{lem}
\begin{proof}
A twin-triangular derivation $\partial=x^{n}\partial_{y}+p_{1}(y)\partial_{z_{1}}+p_{2}(y)\partial z_{2}$
generates a set-theoretically free $\mathbb{G}_{a}$-action if and
only if $x^{n}$, $p_{1}(y)$ and $p_{2}(y)$ generate the unit ideal
in $A[y,z_{1},z_{2}]$. So $\overline{p}_{1}$ and $\overline{p}_{2}$
are relatively prime and at least one of them, say $\overline{p}_{2}$,
is nonzero. If $\overline{p}_{1}=0$ then $p_{2}$ is necessarily
of the form $p_{2}(y)=c+x\tilde{p}_{2}(y)$ for some nonzero constant
$c$ and so changing $z_{1}$ for $z_{1}+z_{2}$ yields a twin-triangular
derivation conjugate to $\partial$ for which the corresponding polynomials
$p_{1}(y)$ and $p_{2}(y)$ both have non zero residue classes modulo
$x$. More generally, changing $z_{2}$ for $\lambda z_{2}+\mu z_{1}$
for general $\lambda\in\mathbb{C}^{*}$ and $\mu\in\mathbb{C}$ yields
a twin-triangular derivation conjugate to $\partial$ and still satisfying
condition a). So it remains to show that up to such a coordinate change,
condition b) can be achieved. This can be seen as follows : we consider
$\mathbb{A}^{2}$ embedded in $\mathbb{P}^{2}={\rm Proj}(\mathbb{C}[u,v,w])$
as the complement of the line $L_{\infty}=\left\{ w=0\right\} $ so
that the coordinate system $\left(u,v\right)$ on $\mathbb{A}^{2}$
is induced by the rational projections from the points $\left[0:1:0\right]$
and $\left[1:0:0\right]$ respectively. We let $C$ be the closure
in $\mathbb{P}^{2}$ of the image of the immersion $j:\mathbb{A}^{1}={\rm Spec}(\mathbb{C}[y])\rightarrow\mathbb{A}^{2}$
defined by integrals $\overline{P}_{1}$ and $\overline{P}_{2}$ of
$\bar{p}_{1}$ and $\bar{p}_{2}$, and we denote by $a_{1},\ldots,a_{r}\in C$
the images by $j$ of the points in the inverse image of the branch
locus of $\overline{P}_{1}:\mathbb{A}^{1}\rightarrow\mathbb{A}^{1}$.
Since the condition that a line through a fixed point in $\mathbb{P}^{2}$
intersects transversally a fixed curve is Zariski open, the set of
lines in $\mathbb{P}^{2}$ passing through a point $a_{i}$ and tangent
to a local analytic branch of $C$ at some point is finite. Therefore,
the complement of the finitely many intersection points of these lines
with $L_{\infty}$ is a Zariski open subset $U$ of $L_{\infty}$
with the property that for every $q\in U$, the line through $q$
and $a_{i}$, $i=1,\ldots,r$, intersects every local analytic branch
of $C$ transversally at every point. By construction, the rational
projections from $\left[0:1:0\right]$ and an arbitrary point in $U\setminus\{\left[0:1:0\right]\}$
induce a new coordinate system on $\mathbb{A}^{2}$ of the form $\left(u,\lambda v+\mu u\right)$,
$\lambda\neq0$, with the property that none of the $a_{i}$, $i=1,\ldots,r$,
is contained in the inverse image of the branch locus of the morphism
$\lambda\overline{P}_{2}+\mu\overline{P}_{1}:\mathbb{A}^{1}\rightarrow\mathbb{A}^{1}$.
Hence changing $z_{2}$ for $\lambda z_{2}+\mu z_{1}$ for a pair
$(\lambda,\mu)$ corresponding to a general point in $U$ yields a
twin-triangular derivation conjugate to $\partial$ and satisfying
simultaneously conditions a) and b). 
\end{proof}

\subsection{Complement : a criterion for properness of twin-triangular $\mathbb{G}_{a}$-actions}

\indent\newline\noindent In contrast with the set-theoretic freeness
of a $\mathbb{G}_{a}$-action on an affine variety, which can be easily
decided in terms of the corresponding locally nilpotent derivation
$\partial$ of its coordinate ring, it is difficult in general to
give effective conditions on $\partial$ which would guarantee that
the action is proper. However, for twin-triangular derivations, we
derive below from our previous descriptions a criterion that can be
checked algorithmically. 

\begin{parn} \label{par:Prop-Crit-setup} For a set-theoretically
free twin-triangular $\mathbb{G}_{a}$-action on the affine space
$\mathbb{A}_{X}^{3}={\rm Spec}(A[y,z_{1},z_{2}])$ over a Dedekind
domain $A$, properness is equivalent to the separatedness of the
algebraic space quotient $Y=\mathbb{A}_{X}^{3}/\mathbb{G}_{a}$. Since
$X$ is affine, the separatedness of $Y$ is equivalent to that of
the morphism $\theta:Y=\mathbb{A}_{X}^{3}/\mathbb{G}_{a}\rightarrow X$
induced by the invariant projection ${\rm pr}_{X}:\mathbb{A}_{X}^{3}\rightarrow X$.
The question being local in the Zariski topology on $X$, we may reduce
again to the case where $A$ is a discrete valuation ring with uniformizing
parameter $x$. 

We may further assume that the corresponding twin-triangular $A$-derivation
$\partial=x^{n}\partial_{y}+p_{1}(y)\partial_{z_{1}}+p_{2}(y)\partial z_{2}$
of $A[y,z_{1},z_{2}]$ satisfies the hypotheses of Lemma \ref{lem:Bad-Plane-Removal}.
If $n=0$, then $\partial$ generates an equivariantly trivial whence
proper $\mathbb{G}_{a}$-action with $y$ as an obvious global slice.
So we may assume from now on that $n\geq1$. Our assumptions guarantee
that similarly to $\S$ \ref{par:Open-subs-A3} above, an integral $P_{i}\in A[y]$
of $p_{i}$ defines a morphism $P_{i}:\mathbb{A}_{X}^{1}\rightarrow\mathbb{A}_{X}^{1}={\rm Spec}(A[t])$
whose restriction $\overline{P}_{i}$ over the closed point of $X$
is non constant. Passing to the Galois closure $C_{i}={\rm Spec}(R_{i})$
of the finite \'etale morphism obtained by restricting $\overline{P}_{i}$
over the complement $C_{0,i}\subset{\rm Spec}(\mathbb{C}[t])$ of
its branch locus enables as in the proof of Proposition \ref{prop:Restricted-quotient}
the expression of $P_{i}(y)-t\in A\otimes_{\mathbb{C}}R_{i}\left[y\right]$
as 
\begin{equation}
P_{i}(y)-t=S_{1,i}(y)\prod_{\overline{g}\in G_{i}/H_{i}}(y-\sigma_{\overline{g},i})+x^{n}S_{2,i}(y)\label{eq:decomp}
\end{equation}
for suitable elements $\sigma_{\overline{g},i}\in A\otimes_{\mathbb{C}}R_{i}$,
$\overline{g}\in G_{i}/H_{i}$ and polynomials $S_{1,i},S_{2,i}\in A\otimes_{\mathbb{C}}R_{i}\left[y\right]$.
Then we have the following criterion:

\end{parn}
\begin{prop}
\label{prop:Proper-Crit} With the assumption and notation above,
the following are equivalent:

a) $\partial$ generates a proper $\mathbb{G}_{a}$-action on $\mathbb{A}_{X}^{3}$,

b) For every $i\neq j$ in $\left\{ 1,2\right\} $ and every pair
of distinct elements $\overline{g},\overline{g}'\in G_{i}/H_{i}$,
$P_{j}(\sigma_{\overline{g},i})-P_{j}(\sigma_{\overline{g}',i})\in A\otimes_{\mathbb{C}}R_{i}$
can be written as $x^{n-k}\tilde{f}_{ij,\overline{g}\overline{g}'}$
where $1\leq k\leq n$ and where $\tilde{f}_{ij,\overline{g}\overline{g}'}\in A\otimes_{\mathbb{C}}R_{i}$
has invertible residue class modulo $x$. \end{prop}
\begin{proof}
The hypothesis on $\partial$ guarantees that the $A$-derivations
$\partial_{i}=x^{n}\partial_{y}+p_{i}(y)\partial_{z_{i}}$ of $A[y,z_{i}]$
are both irreducible. Letting $V_{\partial_{i}}$ be the invariant
open subset of $\mathbb{A}_{X}^{2}={\rm Spec}(A[y,z_{i}])$ associated
to $\partial_{i}$ as defined in $\S$ \ref{par:Open-subs-A3}, it follows
from the construction given in the proof of Proposition \ref{prop:Restricted-quotient}
that $W_{i}=V_{\partial_{i}}\times_{C_{0,i}}C_{i}$ is the total space
of a $\mathbb{G}_{a}$-bundle $\eta_{i}:W_{i}\rightarrow Z_{i}$ over
an appropriate scheme $Z_{i}$. The $\mathbb{G}_{a}$-action on $V_{\partial_{i}}\times{\rm Spec}(\mathbb{C}[z_{j}])\subset\mathbb{A}_{X}^{3}$,
$j\neq i$, induced by the restriction of $\partial$ lifts to one
on $W_{i}\times\mathbb{A}^{1}$ commuting with that by translations
on the second factor and so the quotient $W_{i}'=W_{i}\times\mathbb{A}^{1}/\mathbb{G}_{a}$
has the structure of a $\mathbb{G}_{a}$-bundle $\eta_{i}':W_{i}'\rightarrow Z_{i}$
over $Z_{i}$. Since $\partial$ satisfies the conditions of Lemma
\ref{lem:Bad-Plane-Removal}, it follows from Corollary \ref{cor:Affine-extended-quotient}
and the proof of Proposition \ref{prop:Twin-Loc-trivi} that the properness
of $\partial$ is equivalent to the separatedness of the schemes $W_{i}'$,
$i=1,2$. So it is enough to show that in our case condition b) above
is equivalent to that in Theorem \ref{thm:Aff-criterion}. We only
give the argument for $W_{1}'$, the case of $W_{2}'$ being similar.
With the notation of the proof of Proposition \ref{prop:Restricted-quotient},
$\eta_{1}:W_{1}\rightarrow Z=Z_{1}$ is the $\mathbb{G}_{a}$-bundle
with local trivializations $W_{1}\mid_{Z_{\overline{g}}}\simeq Z_{\overline{g}}\times{\rm Spec}(\mathbb{C}[u_{\overline{g}}])$,
where $u_{\overline{g}}=x^{-n}(y-\sigma_{\overline{g},1})$, $\overline{g}\in G_{1}/H_{1}$,
and transition isomorphism over $Z_{\overline{g}}\cap Z_{\overline{g}'}\simeq{\rm Spec}(A_{x}\otimes_{\mathbb{C}}R_{1})$
given by $u_{\overline{g}}\mapsto u_{\overline{g}'}=u_{\overline{g}}+x^{-n}(\sigma_{\overline{g},1}-\sigma_{\overline{g}',1})$
for every pair of distinct elements $\overline{g},\overline{g}'\in G_{1}/H_{1}$.
The lift to $W_{1}\times\mathbb{A}^{1}$of the induced $\mathbb{G}_{a}$-action
on $V_{\partial_{1}}\times{\rm Spec}(\mathbb{C}[z_{2}])\subset\mathbb{A}_{X}^{3}$
coincides with the one defined locally on the open covering $\{W_{1}\mid_{Z_{\overline{g}}}\simeq Z_{\overline{g}}\times\mathbb{A}^{1},\;\overline{g}\in G_{1}/H_{1}\}$
of $W_{1}\times\mathbb{A}^{1}$ by the derivations $\partial_{\overline{g}}=\partial_{u_{\overline{g}}}+\varphi_{2}(u_{\overline{2}})\partial_{z_{2}}$
of $A\otimes_{\mathbb{C}}R_{1}[u_{\overline{g}},z_{2}]$ where $\varphi_{2}(u_{\overline{g}})=p_{2}(x^{n}u_{\overline{g}}+\sigma_{\overline{g},1})$,
$\overline{g}\in G_{1}/H_{1}$. Letting $\Phi_{2}(t)\in A\otimes_{\mathbb{C}}R_{1}\left[t\right]$
be an integral of $\varphi_{2}(t)\in A\otimes_{\mathbb{C}}R_{1}\left[t\right]$,
a direct computation of invariants shows that $\eta_{1}':W_{1}'=W_{1}\times\mathbb{A}^{1}/\mathbb{G}_{a}\rightarrow Z$
is the $\mathbb{G}_{a}$-bundle with local trivializations $W_{1}'\mid_{Z_{\overline{g}}}\simeq Z_{\overline{g}}\times{\rm Spec}(\mathbb{C}[v_{\overline{g}}])$
where $v_{\overline{g}}=z_{2}-\Phi_{2}(u_{\overline{g}})$, $\overline{g}\in G_{1}/H_{1}$,
and transition isomorphisms 
\[
v_{\overline{g}}\mapsto v_{\overline{g}'}=v_{\overline{g}}+\Phi_{2}(u_{\overline{g}})-\Phi_{2}(u_{\overline{g}'})=v_{g}+x^{-n}(P_{2}(\sigma_{\overline{g},1})-P_{2}(\sigma_{\overline{g}',1})).
\]
So condition b) above for $i=1$ and $j=2$ is precisely equivalent
to that of Theorem \ref{thm:Aff-criterion}. \end{proof}
\begin{rem}
With the notation of $\S$ \ref{par:Prop-Crit-setup}, for every regular
value $\lambda_{i}$ of $\overline{P}_{i}:\mathbb{A}^{1}\rightarrow\mathbb{A}^{1}$,
the expression \ref{eq:decomp} specializes to one of the form 
\[
P_{i}(y)-\lambda_{i}=\overline{S}_{1,i}(y)\prod_{\overline{g}\in G_{i}/H_{i}}(y-\overline{\sigma}_{\overline{g},i})+x^{n}\overline{S}_{2,i}(y)
\]
for elements $\overline{\sigma}_{\overline{g},i}\in A$, $\overline{g}\in G_{i}/H_{i}$,
reducing modulo $x$ to the distinct roots of $\overline{P}_{i}(y)-\lambda_{i}\in\mathbb{C}[y]$,
and polynomials $\overline{S}_{1,i},\overline{S}_{2,i}\in A\left[y\right]$.
One checks that condition b) in Proposition \ref{prop:Proper-Crit}
can be equivalently rephrased in this context as the fact that for
every $i\neq j$ in $\left\{ 1,2\right\} $, every regular value $\lambda_{i}$
of $\overline{P}_{i}$, and every pair of distinct elements $\overline{g},\overline{g}'\in G_{i}/H_{i}$,
$P_{j}(\overline{\sigma}_{\overline{g},i})-P_{j}(\overline{\sigma}_{\overline{g}',i})\in A\setminus x^{n}A$.
This alternative form enables to quickly decide that certain twin-triangular
derivations give rise to improper $\mathbb{G}_{a}$-actions. For instance,
consider the family of derivations $D_{n}=x\partial_{y}+2y\partial_{z_{1}}+\left(1+y^{n}\right)\partial_{z_{2}}$,
$n\geq1$, of $\mathbb{C}\left[x\right]_{(x)}[y,z_{1},z_{2}]$. If
$n=2m$, one has $P_{1}=y^{2}$ and $P_{2}=y\left(y^{2m}+2m+1\right)/(2m+1)$.
At the regular value $0$ of $\overline{P}_{2}$, the $2m$ nonzero
roots of $P_{2}$ come in pairs $\pm\alpha_{k}\in\mathbb{C}^{*}$,
$k=1,\ldots,m$, and so $P_{1}(\alpha_{k})-P_{1}(-\alpha_{k})=0$
for every $k$. It follows that the corresponding action is improper.
In contrast, if $n$ is odd then the criterion is satisfied at the
regular value $0$ of $\overline{P}_{2}$. Actually, for all odd $n$,
it was established in \cite{Deveney2002} by different methods that
the corresponding $\mathbb{G}_{a}$-action is a translation.

For a triangular derivation $\partial=x^{n}\partial_{y}+p_{1}(y)\partial_{z_{1}}+p_{2}(y,z_{1})\partial_{z_{2}}$
of $A[y,z_{1},z_{2}]$ generating a set-theoretically free $\mathbb{G}_{a}$-action
and such that the induced derivation $x^{n}\partial_{y}+p_{1}(y)\partial_{z_{1}}$
of $A[y,z_{1}]$ is irreducible, on can still deduce from Theorem
\ref{thm:Aff-criterion} a more general version of the above criterion
which is again a necessary condition for properness. While more cumbersome
than the twin-triangular case, the criterion can be used to construct
improper actions and has potential to study arbitrary proper triangular
actions. 
\end{rem}
\bibliographystyle{amsplain}

\end{document}